\documentclass{journal}

\usepackage{hyperref}

\title{Bellman's Forest Problem and Computability}

\author{Jacob Canel\\ Department of Mathematics, Pennsylvania State University}
\givenname{Jacob}
\surname{Canel}
\address{Department of Mathematics, The Pennsylvania State University, McAllister Building University Park, State College, PA 16802, United States}
\email{jmc8684@psu.edu}
\urladdr{}


\newtheorem{thm}{Theorem}[section]    % Standard theorem environment with numbering consecutive
\newtheorem{lem}[thm]{Lemma}          % Lemma environment with numbering 
\theoremstyle{definition}
\newtheorem{dfn}[thm]{Definition}
\newtheorem{example}[thm]{Example}
\newcommand{\N}{\mathbb{N}}

\begin{document}

\begin{abstract}    % type your abstract below

The goal of this paper is to refine methods in computable analysis and to employ them in the study of solutions of an optimization problem posed by Bellman \cite{bell}. This problem asks how to find optimal (shortest) paths which do not fit into given plane figures. We show that each instance of Bellman's problem has an arbitrarily small uniformly computable perturbation for which the minimal length of path is computable, and thus the class of optimal paths is a $\Pi_1^0$ class, and there are uniformly computable paths which escape with arbitrarily small length greater than the minimum possible.

\end{abstract}

\maketitle

%%%%%%%%%%%%%%%%%%%%   Start of main body of article
\section{Introduction}
\label{sec:intro}
We begin by paraphrasing Finch and Wetzel \cite{finch2004lost}
\begin{center}
    \textbf{Suppose that you are lost in a forest. You have a map, and so you know the shape and size of the forest. You do not have a compass, GPS, or any other way of determining where you are or what direction you are facing. What path do you take to guarantee the shortest time before you escape the forest?}
\end{center}

This is a geometric optimization problem posed by Bellman in 1956 \cite{bell}. It has faced a fair bit of scrutiny, and optimal solutions have been found for a few classes of shape. Much of the interest in this problem comes from the relation to Moser's Worm Problem \cite{moser1966poorly}, a curved cousin of the Kakeya problem \cite{kakeya}, but there has also been independent interest. It is, for example, one of Williams "Million Buck Problems" \cite{williams2002million}. 

For a family of strongly convex shapes, called the ``fat forests," the best path is simply the diameter of the shape itself (see Finch and Wetzel again \cite{finch2004lost}). For everything else, shapes that are sufficiently `thin' or nonconvex altogether, the problem is still mostly up in the air. The first example of a `thin' shape is a sufficiently long rectangular region (or simply an infinitely long one). 

The optimal path for such a region is a combination of lines and circular arcs which was first discovered by Zalgaller in \cite{Zalgaller1} and rediscovered repeatedly since then. It is affectionately known as the `broadworm' and the 'caliper'. We will use the former. The broadworm is a small outward perturbation of the path formed by two sides of an equilateral triangle with side length $1$, where the small side of the rectangle is also assumed to have length $1$.

A solution is not known for an equilateral triangle, although two notable attempts are zigzagging paths made of three line segments by Gross and Besicovitch, the latter of which is conjectured to be the optimal solution.

What has thus far been missing from our discussion is the existence (or not) of a general algorithm which can be presented with the data of a shape and output the data of a solution. This is where our part of the story starts. We will examine the problem from the point of view of computability theory, in particular to establish three things:

\begin{itemize}
    \item What constitutes a reasonable amount of information for a computer to be supplied to specify an instance of the problem?
    \item What constitutes a reasonable amount of information to specify a solution to the problem?
    \item What is it possible, impossible, or unlikely to be possible for a computer to do to turn answers to the former of these two into answers for the latter?
\end{itemize}

Along the way we will pick up and develop tools and arguments for the analysis of optimization problems in computably presented compact metric spaces, as well as some necessary geometric language for the description of the problem.

In section \ref{sec:two}, we will form the problem analytically, starting by specifying the class of forests under consideration, we then build an auxilliary objective function \hyperref[dfn:esc]{$\mathrm{esc}$} which will help us to discuss solutions. 

In section \ref{sec:three}, we discuss computable analysis. We establish effective systems to describe planar regions and paths, how to effectivize the structures from section two, and an \hyperref[thm:effArAs]{effective version} of the Arzela-Ascoli theorem.

In section \ref{sec:four}, we combine the results of the previous two sections to discuss the computability of the minimal path from the data of the forest, as well as counterexamples to possible extensions of it.

Section \ref{sec:moser} comprises our discussion of the Worm problem, how it may be presented in our terms, and the computability of the relevant quantities.

\subsection{Acknowledgments}
The author wishes to thank Prof. Linda Westrick for her guidance during the early parts of this project and to thank Prof. Jan Reimann for his advice on pulling it together towards the end. Additionally, many thanks are owed to the attendees of the Logic Seminar at Penn State for listening to the contents within at varying levels of coherency.

\section{The Bellman Forest Problem}
\label{sec:two}
\subsection{Forests}
\label{sec:forests}
We start by defining our terms:
\begin{dfn}[Forests and Paths]
A forest in $\R^n$ is a compact subset $A\subset\R^n$. A path in $\R^n$ of length $\leq L$ is an equivalence class of $L$-Lipschitz maps $f:[0,1]\rightarrow \R^n$ up to reparameterization, rotation, and translation. The set of such paths is denoted by $P^L$
\end{dfn}
\label{dfn:forpath}
This means that we will only be considering compact forests, but we will be considering them in arbitrary dimension. The restriction to compact forests is necessary for the problem to be presentable to a computer for arbitrary shape. Intuitively, this is because computers can only accept finite amounts of information at a time, and thus finite information must be able to specify the shape up to some finite error, which is equivalent to the $\epsilon$ net criterion for compactness. Additionally, unless restrictions are made on the structure of an unbounded forest, there need not be a finite solution to the problem at all. 

Our definition of a path requires less justification: all paths of length $\leq L$ are parameterizable as such a function. We could work with the rectifiable curves in $\R^n$ as our path space, but this is more troublesome to use for computations. The rules of the problem dictate that we not know our starting position or direction, and this is encoded by taking equivalence classes up to rotation and translation.

\begin{dfn}[Escape]
A map $f:[0,1]\rightarrow \R^n$ is said to escape a forest if its image intersects the boundary or the exterior. An equivalence class $[f]\in P^L$ is said to escape if every constituent path does.
\end{dfn}
\label{dfn:escape}

The \textit{Bellman Forest Problem} asks, given a forest $A$, what is the path of minimal length which escapes from it?

\begin{example}[The Disk]
If we take $A = \mathbb{D}\subset \R^2$, then the minimal escaping path is the equivalence class of the straight line segment of length $2$.
Suppose we have some path strictly shorter than this, and let $f$ be a representative. Translate $f$ so that $f(0.5)$ is the origin. No point in $[0,1]$ is more than $0.5$ away from $0.5$, and so no point in the image can be a distance more than $1$  from the origin. A similar argument shows that the line is unique among the paths of length 2 in that it achieves that maximal distance, and thus is the unique solution to the circular forest.
\end{example}
\label{example:disk}
The same argument will work for a ball of any dimension.

\begin{example}[Bellman's Examples]
Bellman's original question was about two particular shapes, namely an infinitely long rectangular region, and the half plane, with the latter having a guaranteed starting distance from the boundary. Both of these have been solved. Zalgaller solved the first \cite{Zalgaller1} and Isbell the second \cite{isbell1957optimal}.

At first glance, neither of these would seem to be good forests, as they are both non-compact. However, both of these forests have large translation symmetry group, essentially reducing the set of start points to either a single point in the latter case or a compact interval in the former. 

\end{example}
\label{example:bell}

\subsection{Attempts and Attempt Rating}
\label{sec:twopointone}
\begin{dfn}[Metric On Paths]
We denote by $d([f],[g])$ the metric $$\min_{\phi \in \mathrm{Iso}^+(\R^n)}d_{Haus}(\mathrm{im}(f),\mathrm{im}(g))$$ on paths $[f]$ and $[g]$. 
\end{dfn}
\label{dfn:pathmetric}
This metric renders the set $P^L$ compact and gives a filtration on the set of all finite length paths by the number $L$.

\begin{dfn}
We define the signed distance function $\mathrm{dist}^+(A,\partial A, x)$ by
$$\mathrm{dist}^+(A,\partial A, x) = \begin{cases}
    \mathrm{dist}(x,\partial A) & x\notin A\\
    -\mathrm{dist}(x,\partial A) & x\in A
\end{cases}$$
\end{dfn}
\label{dfn:distplus}

This function measures how far inside (negative) or outside (positive) a point is with regards to the forest $A$. It is $1$ Lipschitz with respect to the point $x$. 

We represent our shape by this function. In the case that $A$ is the closure of its interior, this is uniformly computationally equivalent to a pair of enumerations of the interior and exterior of $A$ as unions of basic open sets. This case covers most situations of interest, such as the case of a finite union of polytopes. Additionally, the case where boundary points are isolated from the interior is not relevant to the original problem, as touching the boundary counts as escape. If this restriction is not placed, the signed distance function contains strictly more information than such an enumeration. Furthermore, we contend that "given the shape" should mean "given (as an oracle) whether each point is inside or outside of the forest, and by how much."

We view a path in the plane with a defined starting point and direction as an instantiation of an attempt to solve the problem, e.g. a possible outcome of a lost hiker trying an element of $P^L$. It will be necessary to quantitatively measure the success or failure of such an outcome.

\begin{dfn}For a function $f:[0,1]\rightarrow \R^n$, we denote 
$$\mathrm{con}(A,\partial A, f) = \max_{x\in[0,1]}\mathrm{dist}^+(A,\partial A, f(x))$$
\end{dfn}
\label{dfn:con}
This function measures how well a hypothetical hiker following $f$ would do, where a higher number means more successfully escaping the forest. It is, in other words, the function which, if maximized among paths of a given length, would show a non-lost hiker how to get as far out of the forest as possible. Of course, such a maximal path (of sufficient length) in this case will always be the straight line connecting the start point to the nearest boundary point. It is nonetheless important that we consider non-straight-line paths, as heuristically, while a non-lost hiker can know where to go, the lost hiker may need to 'rule out' areas of forest by taking strange paths as they go. 

For example, in a forest given as the union of a long rectangle and a circle of radius much larger than the width but much smaller than the length of the rectangle, it may be logical to try something like the broadworm first, walk the diameter of the circle, and then try the broadworm again.  Now that we have a quantitative measure of the outcome of following a path, we need to know how well (or badly) an equivalence class of paths will perform. 

\begin{dfn} We define the worst-case-containment function
$$\mathrm{wcc}(A,\partial A, [f]) = \min_{\sigma\in \mathrm{Isom}^+(\R^+)} \mathrm{con}(A,\partial A, \sigma\circ f)$$
\end{dfn}
\label{dfn:wcc}
This function measures the outcome of an ideal adversarial player choosing a starting place and direction after the hiker has chosen their path. $wcc$ is positive on $[f]$ in $P^L$ iff $[f]$ escapes. With that fact in mind, we make one last definition.

\begin{dfn} We define the escape function
$$\mathrm{esc}(A,\partial A ,L) = \max_{[f]\in P^L}\mathrm{wcc}(A,\partial A, [f])$$
\end{dfn}
\label{dfn:esc}
The escape function is the answer to the question "how well can we do with a path of length at most $L$?" As each of these functions was constructed via alternating minimum and maximum operators over a series of compact spaces from a $1$-Lipschitz function, each successive function is also $1$-Lipschitz.

Furthermore, we will prove later that since $P^L$ is also a computably presented compact metric space, the function $\mathrm{esc}(A,\partial A, L)$ is a computable function. It is also monotone nondecreasing.

We can now rephrase the Forest problem as follows:

\begin{center}
    For a given forest $A$, find the minimal $L$ such that $\mathrm{esc}(A,L) = 0$, and find a witnessing path of this length.
\end{center}

In other words, we must find the smallest $L$ for which $\mathrm{wcc}$ is not entirely negative on $P^{L}$, and then must find an element of $P^L$ which has $\mathrm{wcc}(A,\partial A,[f]) = 0$.

In the absence of scrutiny, it would now appear that all we need to do is calculate this $L$, which for a computable increasing function should be computable. We could then simply search this space until we find a path that works. However, as we do not necessarily have that $esc$ is strictly increasing, just that it is nondecreasing, this is not the case. There are also difficulties with the "search phase" of the problem, which amount to the dual problems of computable closed sets of a computable metric space not having computable elements, and the fact that it is more difficult to find the $argmin$ of a function than $min$.

\section{The Effective Arzela-Ascoli Theorem}
\label{sec:three}
In the interest of rigorizing the statements about computability in the previous section, we now turn to a discussion of computable analysis. In particular, we have to discuss how we might represent analytic objects (which as a rule exist in continuum sized metric spaces) to a Turing machine, and make inferences about those representations.

\subsection{Naming Systems}
\label{sec:naming}
\begin{dfn}
    A name for a real number $\R$ is a sequence $(a,b):\N\rightarrow \Q^2$ such that $a_n < a_{n+1} < b_{n+1} < b_n < a_{n}+2^{-n} $
\end{dfn}
\label{dfn:names}

Given two names for different real numbers $(a ,b)$ and $(c,d)$, we will eventually have $b < c$ or $d < a$. So we will be able to tell from this data when two real numbers are different. On the other hand, if this has not happened by the $n$-th term in the sequence, the two numbers must be within $2^{1-n}$ of one another. However one cannot tell from any finite portion of a pair of names if they do in fact equal each other outright.

This is an important general fact about computable analysis: one can estimate numerical quantities and thus learn if they (or objects which correspond to them) are quantitatively similar or different, and to what degree, but the qualitative question of absolute equality is not something that can be reasoned about in a finitistic way.

\begin{dfn}[Computably Presented Compact Metric Spaces]
    A presentation of a compact metric space $X$ is a collection of the following data:
    \begin{itemize}
        \item An increasing sequence $A_n$ of isometric copies of finite subsets of the metric space, such that every element of $X$ is within $2^{-n}$ of $A_n$
        \item For each pair in $A_n$, a name for the distance between its elements.
        
    \end{itemize}
    A compact metric space is called "Computably Presented" if these Data are computable. We will refer to the space $\{x_n\in\prod A_n: d(x_{n+1}, x_{n} < 2^{-n}\}$ as $\hat{X}$.
\end{dfn}
\label{dfn:cpmetricspace}
\begin{dfn}
    A name for an element $p$ in a metric space $(X,d)$ is a sequence $p_n \in X$ with $d(p,p_n) < 2^{-n}$. If we are fixing a computable presentation of that metric space, we shall assume $p_n\in A_n$ and thus $\hat{X}$
\end{dfn}
\label{dfn:nameinmet}

There is a natural metric, called the Gromov-Hausdorff metric, which assigns distances between compact metric spaces, and this description of a presentation is a name for the compact metric space in this metric. The Gromov-Hausdorff metric is itself a computable metric space (although not a compact one). It is seperable with countable dense set of finite metric spaces with all rational distances. Interestingly, the Gromov-Hausdorff space is also geodesic, despite being neither locally compact nor defined locally as a convex subspace of an infinite dimensional Banach space.

\begin{dfn}[Moduli of Continuity]
A function $\epsilon(\delta)$ from $[0,\infty)\rightarrow [0,\infty)$ is a modulus of continuity for a function $f:X\rightarrow Y$ if $d_X(s,t)<\delta$ implies that $d_Y(f(s),f(t))<\epsilon(\delta)$. 
\end{dfn}
\label{dfn:moduli}
As a modulus on a geodesic space may be replaced with the pointwise largest subadditive function below it without changing the set of functions to which it applies, we will mostly be discussing subadditive moduli. Likewise, there may be discontinuous functions that obey a discontinuous modulus or a modulus with $\epsilon(0) > 0$, and we only care about moduli which guarantee continuity, so will only discuss moduli which are continuous and increasing with $\epsilon(0) = 0$.

We restrict ourselves using specific moduli of continuity for the same reason that this is necessary for continuous model theory (see the introduction by Hart \cite{hartcontinuous}): because the Arzela-Ascoli theorem guarantees that functions between compact metric spaces obeying a uniform modulus are a compact family. This fact shows up in a number of places and is essential to our entire argument.

\begin{example}
    The key examples of moduli of continuity are the Holder moduli $\epsilon(\delta) = L (\delta)^d$, which reduce to Lipschitz continuity in the case $d = 1$. In general $d\leq 1$, or else the function with this modulus must be locally constant and the modulus is super additive. 
\end{example}
\label{example:holder}

The Holder moduli are interesting because they encode a notion of relative dimension between two spaces: the image of an n-dimensional (in Hausdorff dimension) compact metric space is at most $n/d$ under a $d$-Holder map. Holder moduli also arise frequently in analysis and its applications, as they are related to quantitative decay conditions on Fourier coefficients. If two moduli apply to functions $f$ and $g$, the composition of these is a modulus applying to $f\circ g$.

\begin{example}
    Let $f:X\rightarrow Y$ be a map between compact metric spaces, and let $\theta(\delta) = \max_{d_X(s,t)\leq\delta}d_y(s,t)$. $\theta$ is clearly a modulus for $f$, but is not clearly subadditive. However the function $\theta'(\delta) = \max\{\theta(\delta), \sup_{a<\delta < b}(\frac{\theta(b)-\theta(a}){b-a}(\delta-a)\}$ is subadditive. We may also add an additional term of $\delta$ if we wish to ensure that the modulus is strictly increasing.
\end{example}
\label{example:empirical]}
Rather than considering functions themselves, we are restricted to working with their lifts to presentations. We must therefore consider moduli relative to presentations rather than metrics in the absolute sense.

\begin{dfn}
    A name for a function $f:A\rightarrow B$ between two metric spaces with presentations $A_n$ and $B_n$ is a sequence of maps $f_n:A_n\rightarrow B_n$ such that $|f_k(a) -f(a)| < 2^{-k}$ for all $a \in A_k$.
\end{dfn}
\label{dfn:funcname}

This is a name for the function in the usual supremum norm for continuous functions. It is notable that while this data is sufficient to uniquely describe any continuous function $A\rightarrow B$, and any such function will have such a description, such a sequence of partial functions need not converge to an actual continuous function. In the purely analytic setting, this is unproblematic, but computability issues do arise.

\subsection{Descent Of Functions}
\label{sec:descent}
\begin{dfn}[Spaces Relative To Moduli]
Let $\epsilon(\delta)$ be a subadditive, strictly increasing function with $\epsilon(0) = 0$ and let $(X,d)$ be a compact metric space. The space $X_\epsilon$ is defined to be $X$ with the metric $\epsilon(d)$. 
\end{dfn}
\label{dfn:relativespaces}

$(X,d)$ is homeomorphic to $(X,\epsilon(d))$ via the identity on $X$, which obeys the $\epsilon$ modulus with respect to these metrics. Every function with modulus $\epsilon$ from $X$ to $Y$ factors through $X_\epsilon$ via composition with a $1$-Lipschitz map.

There is an algorithm, uniform in a presentation $\hat{X}$ of $X$ and the modulus $\epsilon$, which produces a presentation of $X_\epsilon$. The algorithm amounts to reassigning finite strings in the presentation of $X$ to elements of the $A_n$ as appropriate. This enables us to work exclusively with the $1$-Lipschitz modulus in proofs without losing any computability, as long as we remember that we really work relative (in computability) to the modulus in question.

\begin{dfn}[Descent]
Let $f:\hat{X}\rightarrow \hat{Y}$ be a $1$-Lipschitz map between presentation spaces of $X$ and $Y$. We say that $f$ descends if there is a function $g$ such that $p_Y\circ f = g\circ p_X$. This will occur iff $f$ obeys some modulus $\epsilon$ with respect to $\hat{d}_X$ and $\hat{d}_Y$
\end{dfn}
\label{dfn:descentfunc}
The question of descent is an important one, and connects the topology of the underlying spaces with questions about compatibility. In particular, it is possible to check, computably, if $f$ is $1$-Lipschitz at certain scales, i.e. if partial names are mapped to other partial names in distance decreasing manner. Obeying that modulus globally, then, is a universal quantification over a computable predicate, taking in a function as a map of partial names. We should therefore expect that the class of $1$-Lipschitz functions $\hat{X}\rightarrow \hat{Y}$ which descend are a uniform $\Pi^0_1(\hat{d_X}\oplus \hat{d_Y})$ class, e.g. an effectively closed subset of the space of $1$-Lipschitz functions $\hat{X}\rightarrow \hat{Y}$.

This space, as a subspace of the continuous functions from $X\rightarrow Y$ is metrizable using the usual sup metric (called the sup norm in the case that $Y$ is a Banach space. Per Arzela-Ascoli, it is a compact space with this metric.

An effective (computable) version of Arzela-Ascoli follows, and is essential for the core of the Forest Problem algorithm.

A presentation of this space as a $\Pi^0_1(p_X\oplus p_Y)$ subspace of the maps $\hat{X}\rightarrow \hat{Y}$ is a good start, but we do have a problem: $\Pi_0^1$ classes are not valid presentations! More specifically, they can have "dead ends," where it is impossible to extend a partial name to a full one. 

Take, for instance, the space of maps from $[0,1]\rightarrow 2^\N$ with the standard metrics. At the level of partial names, there are nonconstant maps in this family. If we look to make sure that the 1-Lipschitz modulus is respected at the scale $1/8$th. The map which assigns each real number to a sequence of five zeroes and then a million digits in its decimal expansion will be allowed at this scale, but as $[0,1]$ is connected so there can be no function remotely near to this that is continuous, as all continuous functions of this kind are constant. Since this phenomenon occurs at every scale, we cannot simply prune a finite, or even computable, set of dead ends to get a proper presentation.

The fundamental problem here is that the target space of the function is vastly less connected than the source. We can alleviate this problem by making the source less connected or the target more connected. We now turn to the specifics.

\subsection{Effective Arzela-Ascoli}
\begin{thm}
    Let $p_X:\hat{X}\rightarrow X$ and $p_X:\hat{X}\rightarrow X$ be presentations of compact metric spaces $X$ and $Y$, and let $\epsilon$ be a subadditive, increasing, function with $\epsilon(0) = 0$. The space $(X\rightarrow Y)_\epsilon$ of functions from $X$ to $Y$ obeying this modulus is a uniform compact $\Pi^0_1(p_X\oplus p_Y\oplus \epsilon)$ class within the uniformly relatively computably presented metric space $(\hat{X}\rightarrow \hat{Y})$, which gives a uniformly (in $p_X\oplus p_Y\oplus \epsilon$) computably presented compact metric space in the cases that:
    \begin{enumerate}
        \item $X$ is totally disconnected
        \item If $X= [0,1]$ and $Y$ is a convex metric space.
    \end{enumerate}
\end{thm}
\label{thm:effArAs}
\begin{proof}
We work in the $1$-Lipschitz modulus and extend to other cases via the relativization process outlined in the previous section. Let 
$$A = \{f:\hat{X}\rightarrow \hat{Y}:  \hat{d}_{Y}(f(s),f(t)) \leq d_{\hat{X}}(s,t)\} $$
This is a computably presented metric space: for the presentation, as the $1$-Lipschitz condition above ensures that the description of such functions is finitary: that only finitely many natural number data are required to describe such a function to suitable accuracy. It is also clear that $A$ includes all the functions of interest, as any $1$-Lipschitz function on $X$ induces one on $\hat{X}$. However, it may also contain extraneous functions which obey the modulus with respect to $d_{\hat{X}}$ but not $\hat{d}_X$. The metric on this space is computable as the metric on $Y$ is. We now introduce some more sets:

$$A_n = \{f\in A: \forall \sigma_1,\sigma_2\in 2^n, (d_{\hat{X}}(\sigma_1,\sigma_2) < \delta \implies \hat{d}_Y(f(\sigma_1), f(\sigma_2))) \leq\delta\}$$
This is the set of functions which obey the modulus up to scale $2^{-n}$. It is a $\Pi^0_1(p_X\oplus p_Y)$ class, and so is the intersection $\cap_{n\in\N} A_n$ This gives the first part of the theorem, and leaves the cases left to prove.

If $X$ is totally disconnected, then the "compression" function sending everything described by a partial name to its eventually constant sequence is $1$-Lipschitz continuous, and the map obtained by composing a partial $f$ with the compression map gives a full description of a function.

Any partial map from $[0,1]$ to a convex space, obeying the $1$-Lipschitz condition can be extended by simple interpolation of points, and in the limit this extends to a full map obeying the condition. This argument extends to maps from Euclidean $[0,1]^d$ to $[0,1]^n$, but trying to extend it to more complicated domains firmly puts one out of the scope of this work and into the territory of metric geometers. See the famous book by Burago et al \cite{bbi} for more details.
\end{proof}

\section{Solutions to the Forest Problem}
\label{sec:four}
We now discuss the computable analysis of the Escape Function from section 2.2. We note first of all that $\R^n$ is a computably presented metric space, and that any compact forest can be contained in a large ball of radius much greater than the diameter of the forest itself, which is itself a computably presented compact metric space. The signed distance function is our presentation of the space in the first place, and is Turing equivalent in the case that the forest is the closure of its interior to an enumeration of the open sets inside the forest and those outside.

The function $con$ is a maximum of a computable $1$-Lipschitz function over a computably presented compact space ($[0,1]$), as long as $f$ is computable itself. Such a maximum is a computable $1$-Lipschitz function (in this case in the Hausdorff metric on the image) also. The taking of this maximum is uniform in $\mathrm{dist}^+$.

For a computable $f$, to compute $\mathrm{wcc}(A,\partial A, f)$, one needs to recognize two facts. First that the set of rotations $SO(n)$ is compact, and secondly that the set of translations which keep the image of $f$ inside the large enough ball is also compact. Indeed, these transformations together form a computably presented compact space: a closed and bounded subset of some $\R^d$. This search for a minimum is, again, uniformly computable  in $\mathrm{dist}^+$. The resulting function is again $1$-Lipschitz in $f$, and invariant under translations and rotations, and so is a uniformly in $\mathrm{dist}^+$ computable $1$-Lipschitz function in $[f]$.

However, the escape function is simply taking a maximum of $wcc$ over $P^L$, but for $L$ computable, $P^L$ is a computable quotient of a computable metric space: the $L$-Lipschitz functions, which are a computably presented compact space by Theorem 3.10. This maximum is thus also a $1$-Lipschitz function of $L$, which can be taken to have the domain $[0,\mathrm{Diam}(A)]$

Since $P^L$ is a filtration in $L$, $\mathrm{esc}$ is a nondecreasing function as well. We therefore have that $\mathrm{esc}(A,\partial A, L)$ is a nondecreasing $1$-Lipschitz function uniformly computable in $\mathrm{dist}^+$. As mentioned in section 3.2, we now have to contend with a series of possibilities. We denote 
$$L_1 = \sup\{\lambda: \mathrm{esc}(A,\partial A,\lambda) < 0\}$$
$$L_2 = \inf\{\lambda: \mathrm{esc}(A,\partial A,\lambda) > 0\}$$

$L_1$ is uniformly left c.e., and $L_2$ is uniformly right c.e.. We should not,  a priori, expect either of these numbers to be uniformly computable. However they will be computable, and uniformly so, in the case where they happen to coincide. Intuition could tell an interested mathematician that the case where the two coincide is inevitable or unlikely in general. As it turns out, the two will coincide more often than not, but not always.

\subsection{The Ideal Case: Cocountably Often}
We generalize slightly:
\label{sec:cocount}
$$L_1(r,f) = \sup\{\lambda: f(\lambda) < r\}$$
$$L_2(r,f) = \inf\{\lambda:f(\lambda) > r\}$$
for $r\in(\min f, \infty)$ and $f$ a monotone nondecreasing function with $\lim_{x\rightarrow \infty}f(x) = \infty$

Since $\cup_{\min f<r<M}  (L_1(r,f),L_2(r,f)) $ is a bounded set and the intervals are disjoint, the set of $r$ for which the interval has positive measure is at most countable. We should expect, in other words, that any random number we pick should have its $L_2$ and $L_1$ coincide.

It would be reasonable to object that this doesn't tell us anything about the number zero, which is the only $r$ value that has a clear geometric meaning. We can see the geometric meaning of other $r$ values as well by the following construction:

Let $A$ be a forest, and let $A_r = \{x\in\R^n: \mathrm{dist}^+(A,\partial A, x) \leq r\}$. Each $A_r$ is compact. 
\begin{lem}[Levels of the Escape Function Correspond to Perturbations]
$L_1(r,\mathrm{esc}(A,\partial A)) = L_2(r,\mathrm{esc}(A,\partial A))$ iff $L_1(0,\mathrm{esc}(A_r,\partial A)_r) = L_2(0,\mathrm{esc}(A_r,\partial A_r))$
\end{lem}
\label{lem:levels}

\begin{proof}
Let $\mathrm{esc}(A,\partial A, L) < r$. Then for each path of length $L$, there is a positioning of that path such that every point has $\mathrm{dist}^+$ value strictly less than $r$, and thus the path can be contained in $\cup_{s<r}A_s$. Thus $\mathrm{esc}(A_r,\partial A_r, L) < 0$

Let $\mathrm{esc}(A,\partial A, L) > r$, then there is a path of length $L$ which achieves escape value $r+\epsilon$ for some epsilon. No matter how this path is reoriented or translated, some point of it is at least $\epsilon$ away from $A_r$, and thus $\mathrm{esc}(A_r,\partial A_r, L) > \epsilon$. Thus we conclude that the $A\rightarrow A_r$ operation preserves strict monotonicity at $r$. 
\end{proof}

The forests $A_r$ are quite similar to $A$ for small (positive or negative) $r$, approaching it in the Hausdorff distance for positive $r$ (or for negative $r$ in the case that $A$ is the closure of its interior). They are "nicer" in the sense that they satisfy an internal (if $r>0$) or external (if $r<0$) $C^2$ boundary condition.

In other words, because we have considered such general shapes, making no claims of convexity, simple connectedness, or any other regularity beyond simple compactness. The shapes we have thus far discussed could, in principal, be as pathological as one would like, having any dimension or any fundamental group. 

The punishment for this is that we cannot expect good behavior from the shape we arbitrarily select, but the family of shapes it determines by the $r$-perturbations will be well behaved co-countably often.

\begin{lem}
Let $L_1 = L_2$ for some forest $A$, then there is a uniform algorithm which limit computes a path $f$ of length $L_1$ with $\mathrm{wcc}(A,\partial A, [f]) = 0$.
\end{lem}
\label{lem:limit}
\begin{proof}
If $L_1 = L_2$, then the uniform algorithms left enumerating $L_1$ and right enumerating $L_2$ can be successfully combined into a uniform (but not convergent in defective cases!) algorithm computing $L_1$. If $L_1$ is relatively computable, then $P^{L_1}$ is a uniformly relatively computably presented compact metric space, and thus $\mathrm{wcc}$ is a uniformly relatively computable function on $P^{L_1}$. The elements of it which attain the optimal value of zero are a uniform relative $\Pi^0_1$ class, and thus we may uniformly limit-compute an optimal path.
\end{proof}

\begin{thm}[Almost Every Forest Problem is Limit Solvable By A Uniform Algorithm]
We claim now that the forest problem is, except at a countable number of cases in a given family, solvable in the limit.
\end{thm}
\label{thm:almost}

\begin{proof}
    Immediate from the combination of Lemmas \ref{lem:levels} and \ref{lem:limit}.
\end{proof}

As is typical for many problems in geometry, the case where $A\subset \R^n$ is convex is easier to deal with. In particular, we will find that for convex shapes $L_1 = L_2$. For a convex shape $A$, $A$ contains a path $[f]$ iff it contains its convex hull. For $r > 1$ and $f$ nonconstant, a rotation of the convex hull of $ Im(r f)$ contains that of $\mathrm{im}(f)$ in its interior, and thus cannot be contained inside $A$. Since the length of $rf$ is $r$ times that of $f$, if $f$ is optimal for $\mathrm{wcc}$ at length $L_1$, $rf$ escapes and thus $L_2 = L_1$.

\subsection{Can we find an optimal path?}
\label{sec:canwe?}
It is fair to ask if we can do better: is there an algorithm which straightforwardly computes a minimal path from the data of $\mathrm{dist}^+$? If there is only one such path the answer is yes, as a one point $\Pi^0_1$ class has a computable element. However, if the solution is not unique we are dealing, a priori, with a an arbitrary $\Pi_0^1$ class, which need not have a computable element. 

If we take a rectangle with diameter the same as the length of the broadworm for its width, we immediately see that both the diameter and the broadworm are optimal. One possible geometric approach to resolving this would be finding optimal families of solutions, as the broadworm is part of one, and the diameters are part of another, one which becomes optimal at this length of rectangle and one of which stops being optimal at longer lengths. At present, the author cannot make any claims about the nature of the solution set beyond that it is a nonempty $\Pi_0^1$ class uniformly in the signed distance function.

One possible avenue for constructing a solution (or solution for many more families) is to explore the effects that the geometry of the shape has on the geometry of the solution space. Since even a rectangle need not have a unique solution, it seems hopeless to expect that the solution sets for more complicated sets will be well behaved geometrically. However it still seems possible that boundary conditions like the $C^2$ boundary condition or piecewise linearity may result in topological behavior (such as the existence of predictable isolated solutions) which will allow a computer to approximate solutions in a better way.

It is important to note that difficulty of finding a path of optimal length, compared to the relative ease of finding the length itself, is an example of the difference in difficulty between taking the $\mathrm{max}$ of a function and taking its $\mathrm{argmax}$. It is not difficult, knowing the length, to find paths within $\epsilon$ of that length which escape a given forest, which one might regard as the practical question if one were really lost in the woods. What is difficult is finding an approximation of an actually optimal solution to within a guaranteed error.

Finding an $\mathrm{argmin}$ is in general just as difficult as finding a computable element of a $\Pi_0^1$ class.

\begin{example}
 Take $[0,1]$ and enumerate the list of pairs of left and right enumerations of numbers $(a_n,b_n)_e$. We can define a function by an infinite sum

$$f(x) = \sum_{(e,n)\in A} 2^{-n}\max\left(3^{-e}-\frac{a_n+b_n}{2},0\right) $$
where 
$$A = \{(n,e) : n \text { minimal s.t. } b_n(e)-a_n(e)< 2\times3^{-e}\}$$

The minimum value of this function is clearly zero. This is provable from an analysis of the measure of the supports of the summands. Furthermore, this function is computable. Specifying its value  up to an error of $2^{-k}$ requires only truncating the sequence at term $k+1$ and checking up to a step linear in $k$.

However, this function has no computable minima. Any computable real has a convergent left and right enumeration, and therefore returns a positive value for one of the summands.

On the other hand, given an error bound of $\epsilon$, it would be easy to find a real $x$ for which $f(x) <\epsilon$, simply by searching through a dense enough set for a close enough approximation of the function. That said, in this case one would only have to check the values at $O(1/\epsilon)$ sites in $[0,1]$  of $O(\mathrm{log}(n))$ functions to find that nearly optimal representative.

\end{example}
\label{example:argmin}

In the case of finding a minimal path, a priori one would have to search through $O(R^{1/\epsilon})$ many values if they used the presentation of the space given by effective Arzela-Ascoli. It is also possible that more efficient parameterizations of the space exist, and that this is true especially in the convex case where one must only consider convex paths.

\subsection{The Twin Circles}
\label{sec:andtwins}
Thus far we have focused on the cases- many though there are - where the algorithm works. We now discuss a concrete case where $L_1 < L_2$.

The case where $L_1 < L_2$ holds strictly is difficult to imagine. There is always a path $[f]\in P^{L_1}$ with $\mathrm{wcc}(A,\partial A, [f]) = 0$, and thus must touch the edge of the forest, no matter how it is placed. It is difficult to imagine that there may not be an extension or perturbation of arbitrarily small size which escapes to the exterior. Certainly if the path were thickened, it would be impossible to contain it within the forest. It is tempting therefore to try to thicken it by adding small bumps, of size $2^{-n}$ one after the other, defeating attempts at containment as they arise. Unfortunately, while this does greatly reduce the number of potential placements (even taking away the generic examples), it need not eradicate them entirely.

What follows is a somewhat involved analytic proof that for a forest that is a union of two circles of radius $1$, which touch at a point of tangency, $L_1 < L_2$. This shows that there really are such cases, and so the caveat of taking a perturbation is a necessary one.

\begin{lem}
There is an $\epsilon > 0$ such that a piecewise linear, $2+\epsilon$ Lipschitz path with a pair of points that are a horizontal distance 2 apart along the x-axis, then the set of lines indexed by $x$ as $(x-bt,t)$ which intersect it more than once has measure less than one in $x$
\end{lem}
\label{lem:linebound}

\begin{proof}
    Let $f:[0,1]\rightarrow \R^2$ be $(2+\epsilon)$-Lipschitz. And let $T:\R^2\rightarrow \R^2$ be the map sending $(x,y)\rightarrow (x-by,y)$, then the total variation of $\pi_1\circ T\circ f$ is at most $\sqrt{1+b^2}(2+\epsilon)$.

Let $\epsilon = \frac{2b}{1+b^2}$, so that the total variation is now at most $\sqrt{1+b^2}(2+\frac{2b}{2+b^2})$.

This total variation is, for piecewise linear paths, equal to the integral

$$= I \int_\R |\{t:\exists s (f(t) = (x-bs,s)) \}|dx$$

The integrand of which we denote by $N(x)$. We denote the measure of the set $N^{-1}(n)$ by $a_n$.

The integral is equal to $\sum_{n\in\N}na_n$, and we know that $sum_{n \text{ odd}}a_n \geq 2$ as the path covers a distance of $2$, and that the sum $sum_{1 < n \text{ odd}}a_n \leq I/3$.

So $a_1 \geq 2 - I/3$. The limit as $b$ goes to zero of $I$ is $2$, so as $b$ is taken to be arbitrarily small, $a_1$ can be arbitrarily close to being above $4/3$ , and thus the sum $\sum_{n > 1}na_n$ can be made arbitraily low above $2/3$. The sum $\sum_{n > 1}a_n$ is at most half this, and thus can be made arbitrarily low above $1/3$.
\end{proof}

\begin{lem}
    For an arbitrary $r> 0$ There is an $\epsilon > 0$ such that a $2+\epsilon$ Lipschitz piecewise linear path $f:[0,1]\rightarrow \R^2$ with two points a horizontal distance 2 apart along the x-axis may be contained inside a region defined by taking $\{x\in\R^2: \mathrm{dist}(x,[-1,1]) \leq r\}$ and removing the region horizontally between two lines of the form $(x-bt,y+t)$ and $(x+bt,y+t)$, where $b$ is as in the previous lemma and $x$ is within length $3/8$ of $0$.
\end{lem}
\label{lem:outbound}

\begin{proof}
    \textit{Proof}
Place the two points of horizontal distance 2 at $(-1,0)$ and $(1,0)$. Let $\epsilon$ be such that the ellipse with Foci at $-1$ and $1$ and major axis length $2+2\epsilon$ fits in the region described above before the removal. The path fits inside the ellipse and thus this region. To make the removal, pick and $(x,y)$ in the image of the path such that the associated lines intersect the path only once and $|x| < 3/8$, which is possible by the previous lemma. The path crosses each of these lines exactly once at the selected point, and thus must not enter the regions removed. 
\end{proof}
We note here that we have $\epsilon < r$, and that we may actually pick $r$ to be $\frac{2b}{1+b^2}$ as we like. We now let $A$ be a union of two circles of radius one, with their centers on the x axis and tangent at the origin.

\begin{lem}
    For $r= \frac{2b}{1+b^2}$ small enough, the regions described in the previous lemma may be translated to fit inside of $A$
\end{lem}
\label{lem:translation}

\begin{proof}
    Translate to place the intersection of the lines at the point of tangency. The boundary of the region intersects the circles at four other points where the lines meet the inner semi-circles, but do not go outside. Since the origin is within $1/3$ of the center of the horizontally spaced points, the outer edges of the regions have $x$ values no more extreme than $1+r + 1/3$ and y values no more extreme than $r$, and so they do not touch the outer semicircles, and the regions are contained.
\end{proof}

\begin{thm}
    For $A$ described above, $L_1 < L_2$
\end{thm}

\begin{proof}
$L_1$ in this case is clearly exactly 2, as a straight line of length 2 cannot be placed inside the interior of either circle, and thus not inside the interior of $A$. 
However, for sufficiently small $\epsilon$, any piecewise linear $[f]$ fits inside of the $r$-regions, and therefore inside $A$. The piecewise linear paths are dense in $P^L$, and so $\mathrm{esc}(A,\partial A, 2+\epsilon) = 0$, and thus $L_2 \geq 2+\epsilon > 2$.
\end{proof}

\subsection{The Required Perturbation Can Always Be Uniformly Relatively Computable}
\label{sec:alwaysuniformly}
Now that we know it is possible for $L_1$ and $L_2$ to be distinct, and that within a family this is rare, we must discuss the complexity of the set of ill behaved members of a family. What kind of countable set are we looking at here?

First, let us ask for the quantifier complexity of this set (using $dist^+$ as an oracle). Every countable set is $F_\sigma$ but may not be closed, which suggests that they must be $\Sigma^0_2(X)$ for some oracle $X$. The naive formula defining this set is
$\exists q_1 < q_2 (\mathrm{esc}(A,\partial A, q_2) = A,\partial A, q_2) = r)$, which in computable terms is actually the statement:
$$\exists q_1 < q_2  \forall \delta (esc(A,\partial A, q_2) \in (r-\delta, r+\delta)\ni \mathrm{esc}( A,\partial A, q_1) )$$
Where $q_1$, $q_2$, and $r\pm \delta$ are rational numbers. Since $esc$ is relatively computable, that makes this defective set $\Sigma_2^0$ after all, and the sets corresponding to fixed $q_1$ and $q_2$ values $\Pi^0_1$ singletons, which is to say relatively computable numbers.

This means that any perturbation which is $1$-generic relative to $\mathrm{dist}^+$ will be sufficient to push the situation out of the defective set. Since these numbers are computable, they are also effectively measure zero, and so a Martin-Lof random perturbation will also be sufficient.

We have that any perturbation which does not resolve the $L_1 < L_2$ issue is computable in $\mathrm{dist}^+$, but the converse question remains. Is there a computable perturbation which resolves the issue itself?

Suppose not, then every pair $q_1$, $q_2$ either gives a computable number or will, at some point, be ruled out as representing any number at all. We can then computably diagonalize against all of the computable numbers in this list, (in a uniform way), so there is a uniformly computable perturbation, which can be asked to be arbitrarily small, which has $L_1 = L_2$. A uniformly computable perturbation can be used to enforce good behavior after all.

What this means is that given any shape, there is an arbitrarily close shape computable from the original shape which satisfies $L_1=L_2$, and thus has a computable solution length. 

\section{The Moser Worm Problem}
\label{sec:moser}
We now turn our attention to the related Worm Problem of Moser (introduced in \cite{moser1966poorly}). The problem is described colorfully as follows:

\begin{center}
    \textbf{Suppose there is a worm of length one, and that you wish to crush the entirety of its body with one swing of a hammer, but do not know in what position it will contort itself. What is the hammer of minimal area required to smash the worm?}
\end{center}

In our language, this problem asks for the number

$$\alpha = \inf \{m(A): A\subset \R^2 \text{ compact } , \mathrm{esc}(A,\partial A, 1) \leq 0\} $$

It is typical for people studying this problem to consider only convex shapes. Among the reasons why they do this is that the minimum area (among convex shapes) is actually attained by some convex shape, but this may not be the case for nonconvex shapes. We discuss the nonconvex value, as we are best equipped to do so with the technology we have already built up.

A few things here:
\begin{enumerate}
    \item This is an honest infimum, the space of compact subspaces of $\R^2$ is not compact, even if one takes the quotient by isometries of the plane, and even if one then further restricts the total area.
    \item We always have $\mathrm{esc}(\lambda A, \partial \lambda A, \lambda) = \mathrm{esc}(A,\partial A, 1)$
    \item This infimum is less than $\frac{\pi}{4}$, as the circle of diameter one covers all paths of length less than or equal to one.
\end{enumerate}
Again, we write $L_2(A)$ for the quantity $\{\inf\{q\in\Q:\mathrm{esc}(A,\partial A, q) > 0)\}\}$ and $L_1(A)$ for $\{\inf\{q\in\Q:\mathrm{esc}(A,\partial A, q) > 0)\}$. Item two above means that $L_2(\lambda A) = \lambda L_2(A)$

Item two above suggests a simplification of the problem: if $L_1$ for some shape is strictly greater than one, the shape can be scaled by a factor of $1/L_2(A)$, which reduces its area by a factor of $1/L_2(A)^2$. As long as $L_2(A)>0$, this scaling is the smallest member of the orbit under $\R^\times$ which has $\mathrm{esc}(A,\partial A, 1)$. We can thus rewrite the original infimum as

$$\alpha = \inf\left\{ \frac{m(A)}{L_2^2(A)}:A\subset \R^2 \text{ compact } \right\}$$

but this is a scale invariant quantity. Since any compact set in $\R^2$ can be rescaled to fit inside $[0,1]^2$, we again rewrite the expression as

$$\alpha = \inf\left\{ \frac{m(A)}{L_2^2(A)}:A\subset [0,1]^2 \text{ closed } \right\}$$

\begin{lem}
    $\lim_{q \rightarrow 0^+} L_1(A_q) = L_2(A)$.
\end{lem}
\label{lem:onetwoeq}

\begin{proof}
    Let $q > L_2(A)$ , then there exists a path $[f]\in P^q$ and $r > 0$ with $\mathrm{wcc}(A,\partial A, [f]) > r$. Let $0< q < r$. $\mathrm{wcc}(A,\partial A, [f]) > r-q$, and the thus $\lim_{q\rightarrow 0^+} L_1(A_q)\leq L_2(A_q)$

    $L_1(A_r)$, for $r > 0$, is greater than $L_2(A)$, as $A$ itself is contained in the interior of $A_r$, and thus $\mathrm{esc}(A_r,\partial A_r, L_2(A)) < 0 $
\end{proof}

Since $\cap _{q > 0}A_q = A$, $\lim_{q\rightarrow 0^+} m(A_q) = m(A) $ for any positive $\epsilon$, and combining this with the lemma above, we obtain that the values $\frac{m(A_q)}{L_1(A_q)^2} \rightarrow \frac{m(A)}{L_2(A)^2}$ as $q \rightarrow 0^+$.

Since $A_r$ is also compact and is, for a sufficiently small scaling of $A$, contained in $[0,1]^2$, we can once more rewrite our expression for $\alpha $ as

$$\alpha = \inf\left\{\frac{m(A)}{L_1(A)^2: A\subset [0,1]^2 \text{ closed}}\right\}$$

However, $L_1$ itself is a supremum, namely it is $\sup\{q\in\Q: \mathrm{esc}(A,\partial A, q)\}< 0\}$, which makes $L_1(A)^{-2}$ equal to $\inf\{q^{-2}: q\in Q \text{ } \mathrm{esc}(A,\partial A, q)\}$. This, rather tantalizingly, presents $\alpha$ as an infimum over (1) compact subsets of $[0,1]^2$ and (2) negative value arguments of the escape function, which suggests that $\alpha$ is a one quantifier definable quantity, as long as the set of shapes has a countable subfamily which faithfully represents all arbitrarily low values of $\frac{m(A)}{L_1(A)^2}$.

\begin{lem}
There is a computable, and thus countable, set of shapes which have computable $\frac{m(A)}{L_1^2(A)^2}$ values that are dense in the full set of such values.
\end{lem}
\label{lem:inv}
\begin{proof}
    Let $A\subset [0,1]^2$ be compact. We let 

$$ N(i,j, n) = \left[\frac{i}{2^{n}}, \frac{(i+1)}{2^n}\right]\times\left[\frac{j}{2^{n}}, \frac{(j+1)}{2^n}\right]$$
$$ I(A,n) = \left\{ i,j \in [0,2^n]\cap\N: N(i,j,n)\cap A \neq \emptyset\right\}$$
and 
$$B(A,n) = \cup_{((i,j)\in I(A,n)}N(i,j,n)$$

$B(A,n)$ is always computable (just as a shape but also uniformly, if somewhat tautologically, in $A$). Not coincidentally, this the canonical dense subset used in the presentation of the Hausdorff metric space of the square as a computably presented compact metric space.

We have $\cap_{n\in\N} B(A,n) = A$, and so $\lim_{n\rightarrow \infty} m(B(A,n)) = m(A)$ as these sets are nested. Furthermore, they each contain $A$ itself, and thus $L_1(B(A,n)) \geq L_1(A)$. Furthermore, $d_H(B(A,n), A) < 2^{-1-n }$, so for any positive $q$ there is an $N$ with $B(A,n)\subset A_{q}$ for any $n > N$, which means that
$ \lim_{n\rightarrow \infty}L_1(B(A,n)) = \lim_{q\rightarrow 0^+}L_1 (A_q) = L_2(A)$
\end{proof}

We thus only need to consider shapes which are finite unions of closed dyadic squares, and can rewrite $\alpha$ one final time
$$\alpha = \inf\left\{\frac{m(A)}{L_1(A)^2: \text{ $A$ a finite union of Dyadic squares}}\right\}$$

Since the finite unions of dyadic squares are a computable set of computable subsets of $[0,1]^2$, meaning they have computable $dist^+$ functions, the associated $L_1$ values are uniformly left c.e., and their measures are straightforwardly computable, meaning the ratios $\frac{m(A)}{L_1(A)^2}$ are right c.e. and so is $\alpha$, which gives us the final theorem of this paper.

\begin{thm}
    The infimum of possible areas of shapes which cover all paths of length $1$ in the plane is a right c.e. number.
\end{thm}
\label{thm:Moser}

This number, of course, is the answer to the Moser worm problem.
%%%%%%%%%%%%%%%%%%%%   End of main body of article
%
%                             References
%
%   BiBTeX users uncomment the following line:
%
\bibliographystyle{plain}
\nocite{Zalgaller1}
\nocite{Zalgaller2}
\nocite{williams2002million}
\nocite{wetzel2005letter}
\nocite{isbell1957optimal}
\nocite{bell}
\nocite{finch2004lost}
\nocite{moser1966poorly}
\nocite{DH}
\nocite{klausw}
\nocite{bbi}
\bibliography{bibliography}

\end{document}